\documentclass[12pt]{article}
\usepackage{amsfonts,amssymb,latexsym, amsmath, amsthm}

\usepackage[usenames]{color}

\newtheorem{thm}{Theorem}[section]
\newtheorem{lem}[thm]{Lemma}

\newtheorem{defn}{Definition}[section]

\theoremstyle{remark}

\theoremstyle{plain}

\numberwithin{equation}{section}

\def\leq{\leqslant}
\def\geq{\geqslant}
\def\Aut{{\rm Aut}}

\def\RR{{\mathbb R}}

\newcommand{\Om}{\Omega}

\newcommand{\Wreath}{\mbox{ \rm Wr }}
\newcommand{\wreath}{\mbox{ \rm wr }}

\def\Z{{\mathbb Z}}

\def\QQ{{\mathbb Q}}

\def\mC{{\mathcal C}}

\def\fK{{\mathfrak K}}

\def\C{{\mathrm C}}
\def\dep{\mathrm {dep}}
\def\cB{{\mathcal B}}

\newtheorem{defi}{Definition}[section]
\newtheorem{theorem}{\bf Theorem}

\newtheorem{lemma}[defi]{Lemma}
\newtheorem{prop}[defi]{Proposition}

\newtheorem{cor}[defi]{Corollary}

\newtheorem{remark}[defi]{Remark}
\newtheorem{example}[defi]{Example}
\newtheorem{*example}[defi]{*Example}

\addtocontents{ences.txt}{References}

\def\suppo{\mathrm{supp}}
\def\rst{\mathrm{rst}}
\def\St{\mathrm{st}}
\def\Wr{\mathrm{Wr}}
\def\wr{\mathrm{wr}}
\def\mro{\mathrm{o}}

\begin{document}

\title{The first-order theory of  $\ell$-permutation groups} 
\author{A. M. W. Glass and John S. Wilson} \small{\date{\today}} \maketitle { \def\thefootnote{} \footnote{2010 AMS Classification:  06F15, 20B07, 05C05, 03C60.

} }

\setcounter{theorem}{0}
\setcounter{footnote}{0}

\begin{abstract}

Let $(\Omega, \leq)$ be a totally ordered set.  We prove that if $\Aut(\Omega,\leq)$ is transitive and satisfies the same first-order sentences as $\Aut(\RR,\leq)$ (in the language of lattice-ordered groups) then $\Omega$ and $\RR$ are isomorphic ordered sets. This improvement of a theorem of Gurevich and Holland is obtained as one of many consequences of a study of centralizers and coloured chains  associated with certain transitive subgroups of $\Aut(\Omega,\leq)$.  
\end{abstract}

\section{Introduction}  
We recall that a lattice-ordered group (or $\ell$-group for short) is a group that is also a lattice with respect to two additional operations $\vee$, $\wedge$ such that 
$$h(f\vee g)k=hfk\vee hgk,\quad h(f\wedge g)k=hfk\wedge hgk\quad \hbox{for all } f,g,h,k\in G.$$
We write $G_+$ for $\{ g\in G\mid g>1\}$ and $|g|=g\vee g^{-1}$; now $|g|\geq 1$ with equality if and only if $g=1$.  If the order on $G$ is a total order, $G$ is called an o-{\em group}.

Well-studied examples of $\ell$-groups are automorphism groups $\Aut(\Omega,\leq)$ of totally ordered sets $(\Omega,\leq)$.  The operations are pointwise:
that is, the join and meet of elements $f,g$ are defined by $$\alpha(f\vee g)= \max\{\alpha f,\alpha g\}\quad\hbox{and}\quad 
\alpha(f\wedge g)= \min\{\alpha f,\alpha g\}\quad \hbox{for all }\alpha\in \Omega.$$
An $\ell$-{\em permutation group} $(G,\Omega)$ is a sublattice subgroup of $\Aut(\Omega,\leq)$.  Transitive $\ell$-permutation groups are of particular interest, and the groups studied in this paper will always be assumed to be transitive. 

In 1981, Gurevich and Holland \cite{GH} proved the following result.

\begin{thm} \label{GH}  Suppose that $(\Omega,\leq)$ is a totally ordered set such that $\Aut(\Omega,\leq)$ acts transitively on pairs $(\alpha,\beta)$ with $\alpha< \beta$, 
and that $\Aut(\Omega,\leq)$ and $\Aut(\RR,\leq)$ satisfy the same first-order sentences in the language of lattice-ordered groups.  Then $\Omega$ is isomorphic to $\RR$ as an ordered set. \end{thm}

Other results of a similar kind were obtained in \cite{GGHJ}.  
One consequence of Theorem 1.1 and our main result,  Theorem A below, is the following stronger result.

\begin{cor} \label{reals} Suppose that $(\Omega,\leq)$ is a totally ordered set on which $\Aut(\Omega,\leq)$ acts transitively, 
and that $\Aut(\Omega,\leq)$ and $\Aut(\RR,\leq)$ satisfy the same first-order sentences in the language of lattice-ordered groups.  Then $\Omega$ is isomorphic to $\RR$ as an ordered set.
\end{cor}

Transitivity is necessary in the above result.  Let $\Lambda$ be any rigid totally ordered set with at least two elements (for example a finite totally ordered set with at least two elements), and let $\Omega = \Lambda \times \RR$, with the order defined by $(\lambda_1,r_1)< 
(\lambda_2,r_2)$ if $r_1<r_2$ or if both $r_1=r_2$ and $\lambda_1<\lambda_2$.
It is easy to see that $\Aut(\Lambda \times \RR, \leq)$ is isomorphic to $\Aut(\RR,\leq)$.

Another special case of our main result concerns Wreath products (see \cite{HM} or \cite[Chapter 5]{G81} for the definition and properties).

\begin{cor}
\label{sumauto}
Let $A_n:=\Z$ for all $n\in \Z$ and let
$(W,\Lambda)$ be the Wreath product 
$\Wreath\{(A_n,A_n)\mid n\in \Z\}$.
Assume that $\fK$ is a totally ordered set and that $\{(G_K,\Omega_K)\mid K\in \fK\}$ is a family of o-primitive $\ell$-permutation groups. Define $(G,\Omega):=\Wreath\{(G_K,\Omega_K)\mid K\in \fK\}$.
If $G$ satisfies the same first-order sentences as $W$, then each $G_K$ is isomorphic to $\Z$ and $\fK$ is elementarily equivalent to $\Z$ as a totally ordered set.
\end{cor}

With each abelian o-group $G$ one can associate a totally ordered set of convex subgroups of $G$ and countably many `colours' that range over certain fundamental abelian o-groups.  This idea  was due to Gurevich \cite{GU}, who in 1977 showed that two 
abelian o-groups $G, H$ satisfy the same first-order sentences in the language of o-groups if and only if their coloured chains satisfy the same first-order sentences as coloured chains.  Our main result shows that a similar implication is valid for many non-commutative transitive $\ell$-permutation groups. 

\begin{theorem}\label{MTd}
Let $(G,\Omega)$ and $(H,\Lambda)$ be abundant transitive  $\ell$-permutation groups satisfying $(\dagger)$ and suppose that every o-primitive component of $(G,\Om)$ has irreducible elements. If $G\equiv H$ then every o-primitive component of $(H,\Lambda)$ has irreducible elements, and the
spines $\fK_G$ of $G$ and $\fK_H$ of $H$ are elementarily equivalent 
as  totally ordered sets. If in addition $(G,\Omega)$ is controlled, then  the coloured chains $\frak C_G$ $\frak C_H$ associated with $G$, $H$ are elementarily equivalent as coloured chains.
\end{theorem}

The  notions of abundant and controlled $\ell$-permutation groups and the hypothesis $(\dagger)$ are defined in Sections 2, 7, and  5, respectively, irreducible elements are defined in Section 6 and the spine of an $\ell$-permutation group is defined in Section 2; the associated coloured chain is defined in Section 7. Any transitive $\ell$-permutation group  $(\Aut(\Om,\leq),\Om)$ is abundant, and all of its o-primitive components have irreducible elements. 
We will show in Remark \ref{irred1} that if $(G,\Om)$ is abundant and satisfies ($\dagger$), then the property that every o-primitive component has irreducible elements is expressible by an explicit first-order sentence in the language of $\ell$-groups.
In the statement of Theorem \ref{MTd} and throughout the article, the symbol $\equiv$ signifies elementary equivalence between lattice-ordered groups. An example given in Section 7 shows that coloured chains can be elementarily equivalent without the lattice-ordered groups being elementarily equivalent. 

We explain how Corollary \ref{reals} follows from Theorem \ref{MTd}.
A transitive $\ell$-permut\-ation group is o-primitive if and only if its set $\fK$ of covering convex congruences  has only one element.  Thus if $(G,\Omega)$ is abundant and satisfies $(\dagger)$, and if $G$ satisfies the same first-order sentences as ${\rm Aut}(\RR,\leq)$, then $G$ acts o-primitively on $\Omega$ by Theorem \ref{MTd}.   Since non-abelian transitive o-primitive automorphism groups are always abundant, controlled and satisfy $(\dagger)$, it follows from Corollary \ref{opa} and Theorem \ref{GH} that if ${\rm Aut}(\Omega,\leq)$ and ${\rm Aut}(\RR,\leq)$ are elementarily equivalent then
$\Omega$ is isomorphic to $\RR$.

Similar arguments allow us to strengthen other known results.  For example, 
we have   

\begin{cor} \label{GHQ2}  Suppose that $(\Omega,\leq)$ is a totally ordered set on which $\Aut(\Omega,\leq)$ acts transitively.  Let $\Aut(\Omega,\leq)$ and $\Aut(\QQ,\leq)$ satisfy the same first-order sentences in the language of lattice-ordered groups.  Then $\Omega$ is isomorphic to $\QQ$ or $\RR \setminus \QQ$ as an ordered set. \end{cor} 

The corresponding result with the stronger hypothesis that
$\Aut(\Omega,\leq)$ acts transitively on pairs $(\alpha,\beta)$ with $\alpha< \beta$
is proved in \cite{GH}.

In Corollaries \ref{fin} and \ref{Z-} we give further consequences of Theorem \ref{MTd} concerning the first-order theory of o-primitive components.
 
The breakthrough comes from employing a technique in \cite{W}; we use double centralizers of certain subsets of lattice-ordered groups to give first-order expressibility of covering convex o-blocks.

\section{Preliminaries}

Let $(G,\Omega)$ be a transitive $\ell$-permutation group.  A {\em convex $G$-congruence} $\mathcal C$ on $\Omega$ is a $G$-congruence  with all  ${\mathcal C}$-classes convex; these classes are called {\em o-blocks}.  
We suppress the mention of $G$ if it is clear from context.
By transitivity, each o-block $\Delta$ is a class of a unique convex
congruence; its set of classes is $\{ \Delta g\mid g\in G\}$.
We denote this convex congruence by $\kappa(\Delta)$.
  
\begin{remark} \label{ccc} {\rm
The set of convex congruences of a transitive $\ell$-permutation group  $(G,\Omega)$ is totally ordered by inclusion; see \cite[Theorem 3.A]{G81}. Hence the set of o-blocks forms a {\em root system}   under inclusion;  that is, if $\Delta_1$, $\Delta_2$  are o-blocks of possibly different convex congruences and $\Delta_1\cap\Delta_2\neq\emptyset$ then $\Delta_1 \subseteq \Delta_2$ or $\Delta_2 \subseteq \Delta_1$.}
\end{remark}

If ${\mathcal C}$ and ${\mathcal D}$ are convex congruences with ${\mathcal C}$ strictly contained in ${\mathcal D}$ and there is no convex congruence strictly between ${\mathcal C}$ and ${\mathcal D}$, then we say that 
$\mathcal D$ {\em covers} $\mathcal C$ and
${\mathcal C}$ is {\em covered} by ${\mathcal D}$. 

\begin{prop} \label{cc}
{\rm (\cite[Theorem 3D]{G81})}
 Every convex congruence other than the trivial convex congruence $\mathcal E$ is the union of all covering convex congruences that are contained in it, and
every convex congruence other than the universal convex congruence ${\mathcal U}$ is also the intersection of all covered convex congruences containing it.
\end{prop} 

Let $\alpha,\beta\in \Omega$ be distinct.  Then both the union $U(\alpha,\beta)$ of all convex congruences $\mC$ for which $\alpha$, $\beta$ lie in distinct o-blocks and the intersection $V(\alpha,\beta)$ of all convex congruences $\mC$ for which $\alpha$, $\beta$ lie in the same o-block are convex congruences. So $U(\alpha,\beta)$ is covered by $V(\alpha,\beta)$.  Let 
$$\fK=\{V(\alpha,\beta)\mid\alpha,\beta\in\Omega,\alpha\neq\beta\}.$$
Thus $\fK$ is totally ordered by inclusion. 
It is called the {\em spine} of $(G,\Omega)$.  
For all $\alpha,\beta\in \Omega$ we have $\beta=\alpha g$ for some $g\in G$  by transitivity. 
Therefore $\fK$ can also be described as follows:
$$\fK=\{V(\alpha,\alpha g)\mid\alpha\in\Omega,g\in G, \alpha g\neq \alpha\}. $$
Write $T$ for the set of o-blocks of elements of $\fK$. If $\Delta\in
T$, then $\kappa(\Delta)\in \fK$ and so $\kappa$ restricts to a surjective map from $T$ to $\fK$.
For each $\mC \in \fK$, write $\pi(\mC)$ for both the convex congruence covered by $\mC$ and its set of  o-blocks; 
the latter inherits a total order from $\Omega$.
If  $\Delta$ is a $\mC$-class, let $\pi(\Delta)$ be the set of all $\pi(\mC)$-classes contained in $\Delta$.

We define the stabilizer $\St(\Delta)$  and rigid stabilizer $\rst(\Delta)$ of  an o-block $\Delta$ of a convex congruence as follows:
$$\St(\Delta):=\{g\in G \mid \Delta g=\Delta\},\hbox{ and }\quad\rst(\Delta):=  \{ g\in G\mid \suppo(g)\subseteq \Delta\},$$
where $\suppo(g):=\{ \alpha\in \Omega\mid \alpha g\neq \alpha\}.$ 
So $\St(\Delta)$ and $\rst(\Delta)$ are convex sublattice subgroups of $G$ and $\rst(\Delta)\subseteq\St(\Delta)$.

 For each $\Delta\in T$ let
$$ Q_{\Delta}=\{ h\in \rst(\Delta)\mid (\exists \alpha\in \Delta)(V(\alpha, \alpha h)=\kappa(\Delta))\}.$$
We shall say that $(G,\Omega)$ is {\em ample} if  $Q_{\Delta}\neq \emptyset$ for each $\Delta$.  We note that if $G$ is ample and $\Delta\in T$ then each element of $\kappa(\Delta)$ is moved by some element of $Q_\Delta$.

\begin{lemma}\label{neweltsd} 
Suppose that $(G,\Om)$ is transitive and ample.  Let $\Delta\in T$  and $\Gamma\in \pi(\Delta)$. If $\Delta$ is not a minimal element of $T$, then
$\rst(\Gamma)\neq 1$.
\end{lemma}

\begin{proof}  Choose an element 
$\Delta_0$ of $T$ strictly contained in $\Delta$. Since the convex congruences are linearly ordered by inclusion, $\Gamma$ is a union of o-blocks $\Delta_0g$.   For each o-block we have $\Delta_0g\in T$ and
$Q_{\Delta_0g}\subseteq  \rst(\Gamma)$.  \end{proof} 

 Let  $\Delta\in T$. Each $g\in \St(\Delta)$ induces an action $g_\Delta$ on the ordered set $\pi(\Delta)$ given by $$\Gamma  g_\Delta=\Gamma g\quad (\Gamma\in \pi(\Delta),\; g\in \St(\Delta)).$$   Let 
$$ G(\Delta):=\{ g_\Delta\mid g\in \St(\Delta)\}.$$

 Note that $(G(\Delta),\pi(\Delta))$ is transitive and o-primitive.  Furthermore, if $K\in \fK$ and $\Delta, \Delta'$ are both $K$-classes, then $(G(\Delta),\pi(\Delta))$ and 
$(G(\Delta'),\pi(\Delta'))$ are isomorphic, the isomorphism being induced by conjugation by any $f\in G$ with $\Delta f=\Delta'$  since $(\Gamma f)(f^{-1}gf)=(\Gamma g)f$ for all $g\in \rst(\Delta),\;\Gamma\in \pi(\Delta)$.  
It is customary to write $(G_K,\Omega_K)$ for any of these $\ell$-permutation groups; they are independent of $\Delta$ to within $\ell$-permutation isomorphism.
 
Let $\Lambda\subseteq \Omega$ and $g\in G$ with $\Lambda g=\Lambda$.
Write $\dep(g,\Lambda)$ for the element of $\Aut(\Omega,\leq)$ that agrees with
$g$ on $\Lambda$ and with the identity elsewhere; thus
$$\alpha\,\dep(g,\Lambda) =  \begin{cases} \alpha g&\hbox{if }\alpha\in \Lambda \cr \alpha & \hbox{if } \alpha\not\in \Lambda.
\end{cases} $$
Of course, in general the automorphism $\dep(g,\Lambda)$ need not belong to $G$.
We say that $(G,\Omega)$ is {\em abundant}  if for each
  $\Delta\in T$, and each $g\in \St(\Delta)$, we have $\dep(g,\Delta)\in G$. 
In particular, the kernel of the action of $G$ on the set of translates of $\Delta$ then contains the direct sum of the restricted stabilizers of these translates.
If $(G,\Omega)$ is abundant and $\Delta\in T$, then for each $g\in \St(\Delta)$, the elements $g$ and $\dep(g,\Delta)\in \rst(\Delta)$ induce the same action on $\pi(\Delta)$. If $(G,\Om)$ is abundant, then it is ample since if $\delta_1,\delta_2\in \Delta$ with $V(\delta_1,\delta_2)=\kappa(\Delta)$, let $g\in G$ with $\delta_1g=\delta_2$; then $g\in \St(\Delta)$ and $\dep(g,\Delta)\in Q_\Delta$.

We note that if $\Aut(\Omega,\leq)$ acts transitively on $\Omega$ then 
$(\Aut(\Omega,\leq),\Omega)$ is abundant.

For each automorphism $g$ of a totally ordered set $\Omega$ and each $\alpha$ in the support  $\suppo(g)$ of $g$, the
{\em supporting interval $\Lambda(\alpha,g)$ of $g\in G$ containing $\alpha$} is defined to be the smallest convex subset of $\Omega$ containing $\{ \alpha g^n\mid n\in \Z\}$. 

On any supporting interval $\Delta$ of $g$, either $g$ moves every point up, or it moves every point down. Moreover, if $f\in \Aut(\Omega)$ and $g$ moves every point of $\Delta$ up (respectively, down) then so does $f^{-1}gf$ on its supporting interval $\Delta f$. 

We will need the classification of o-primitive transitive $\ell$-permutation groups provided by McCleary's Trichotomy Theorem (S.H.\ McCleary \cite{Mc}, or  \cite[Theorem 4A]{G81} or \cite[Theorem 7E]{G99}). Let $(\bar{\Omega},\leq)$ denote the Dedekind completion of the totally ordered set $(\Omega,\leq)$ with the maximum element and minimum element removed. Assume that there is $t\in \Aut(\bar{\Omega},\leq)_+$ 
such that the supporting interval in $\bar\Omega$ of each $\alpha\in\Omega$ is $\bar\Omega$. 
If in addition $G$ is transitive on $\Omega$ and $\C_{\Aut(\bar \Omega, \leq)}(G)=\langle t\rangle$, then
$(G,\Omega)$ is called {\em periodic} with {\em period} $t$. If $(G,\Omega)$ is periodic, o-primitive and $\alpha\in \Omega$, then for all $\beta,\gamma\in \Omega$ with $\alpha<\beta<\gamma<\alpha t$ there are elements $g\in G_+$ with $\suppo(g)\cap (\alpha,\alpha t)\subseteq (\beta,\gamma)$; see \cite[II]{Mc} or \cite[Theorem 4.3.1]{G81}.  We note that if
$(G,\Omega)$ is  periodic with period $t$, and if the element $g\in G$
satisfies $\alpha g>\alpha t$ for some $\alpha\in \Omega$, then $\bar{\beta}g>\bar{\beta}$ for all $\bar{\beta}\in \bar{\Omega}$. 

If $G$ is transitive on all $n$-tuples $(\alpha_1,\dots,\alpha_n)\in \Omega^n$ with $\alpha_1< \dots < \alpha_n$, we say that $(G,\Omega)$ is o-$n$ {\em transitive}. An o-$2$ transitive $\ell$-permutation group is o-$n$ transitive for all $n\in \Z_+$; see \cite[Lemma 1.10.1]{G81}. 
      
\begin{prop}\label{trichot}
{\rm (S.H.\ McCleary)}

Let $(G,\Omega)$ be a transitive $\mro$-primitive $\ell$-permutation group. Then one of the following holds$:$
\begin{enumerate}
\item[\rm (I)] $(\Omega,\leq)$ is order-isomorphic to a subgroup of the reals and the action of $G$ on $\Om$ is the right regular representation$;$
\item[\rm  (II)] $(G,\Omega)$ is $\mro$-$2$ transitive$;$ 
\item[\rm  (III)] $(G,\Omega)$ is periodic with period $t$ and the restriction of $G$ to $\Omega \cap (\alpha,\alpha t)$ is $\mro$-$n$ transitive for all $\alpha\in \Omega$ and $n\in \Z_+$.
\end{enumerate}
\noindent Moreover, any transitive $\ell$-permutation group satisfying {\rm(i)} or {\rm(ii)} is $\mro$-primitive, but not all periodic $\ell$-permutation groups are $\mro$-primitive.
\end{prop}

We refer to the o-primitive $\ell$-permutation groups $(G,\Omega)$ in (I), (II), (III) above as being of  types (I), (II) and (III) respectively.  Those of type (I) are abelian o-groups, whereas those of type (II) have trivial centre. Those of type (III) have centre $G \cap \langle t\rangle$. 

\begin{cor} \label{opa}
If $(\Aut(\Om,\leq),\Om)$ is transitive and non-abelian, then it is o-$2$ transitive.
\end{cor}

Our notation for conjugates and commutators is in accordance with our use of right actions: we write $g^f$ for $f^{-1}gf$ and $[f,g]$ for $f^{-1}g^{-1}fg$. 

We write $X\subset Y$ for $X\subseteq Y$ and $X\neq Y$.

  For further background and notation, see \cite{HM}, \cite[Part II]{G81} or \cite[Chapter 7]{G99}.

\section{A technical lemma}

\begin{lemma} \label{primd}
Let $(G,\Omega)$ be o-$2$ transitive and $g,h\in G$ with $\suppo(h)\cap \suppo(h^g)=\emptyset$.
Then there are elements $f,k\in G$ such that $$[h^{-1},h^f][h^{-g},h^{gk}]\neq [h^{-g},h^{gk}][h^{-1},h^f].$$
\end{lemma}

\begin{proof} 
Since $\suppo(h)\cap \suppo(h^g)=\emptyset$, 
we may assume without loss of generality that there are supporting intervals $\Delta_1, \Delta_2:=\Delta_1g$ of $h$ and $h^g$, respectively, such that
 $\delta_1<\delta_2$ for all $\delta_i\in \Delta_i$ ($i=1,2$), the proof for $\Delta_2:=\Delta_1g<\Delta_1$ being similar. 
 Without loss of generality, $\delta_1 h>\delta_1$ for all $\delta_1\in \Delta_1$ (and so $\delta_2 h^g>\delta_2$ for all $\delta_2\in \Delta_2$).
Let $\gamma,\delta, \lambda, \mu\in \Delta_2$ with
$$\gamma<\gamma h^g<\mu h^{-g}<\delta<\lambda<\mu<\delta h^g<\lambda h^g.$$
The six elements 
$$\gamma,\;\gamma h^g,\; \mu h^{-g},\;\delta,\;\mu,\;\lambda h^g \eqno(1) $$
constitute a strictly increasing sequence in $\Delta_1$.  Choose $\xi_{-1},\xi_0\in \Delta_1$ with $\xi_{-1}<\xi_0$, and elements $\xi_1,\xi_2\in \Delta_2$ with  
$$\xi_0<\xi_1<\xi_1 h^g<\xi_2<\xi_2 h^g.$$   Then the six elements
$$\xi_{-1},\;\xi_0,\;\xi_1,\;\xi_1 h^g,\;\xi_2,\;\xi_2 h^g \eqno(2)$$ 
constitute a strictly increasing sequence in $\Omega$.  Using o-$6$-transitivity we can find an element $k$ of $G$ that maps the $n$th element of sequence ($2$) to the $n$th element of sequence ($1$) for each $n$.  
Since $\suppo(h)\cap \suppo(h^g)=\emptyset$ and $\xi_{-1}\in \Delta_1\subseteq \suppo(h)$ we have $\gamma h^{gk}=\gamma k^{-1}h^gk=\xi_{-1}h^gk=\xi_{-1}k=\gamma$.  This and other similar easy calculations shows that
$$
\gamma h^{gk}=\gamma,\quad (\gamma h^g)h^{gk}=\gamma h^g,\quad (\mu h^{-g})h^{gk}=\delta,\quad  \mu h^{gk}=\lambda h^g.$$

Now choose $\alpha\in \Delta_1\subseteq \suppo(h)$ and $\beta\in (\alpha h^{-1},\alpha),$ and choose 
$\zeta_1,\dots,\zeta_4\in \suppo(h)$ such that the eight elements 
$$\zeta_4,\;\zeta_3,\;\zeta_2,\;\zeta_1,\;\zeta_4h,\;\zeta_3h,\;\zeta_2h,\;\;\zeta_1h$$
form a strictly increasing sequence. 
Since $\suppo(h)<\suppo(h^g)$, the eight elements 
$$\beta h^{-3},\,\;\alpha  h^{-3},\,\;\beta h^{-2},\,\;\alpha  h^{-2},\,\;\gamma h^{-g},\,\;\gamma ,\,\;\delta,\,\;\lambda$$
also form a strictly increasing sequence, and we can find an element $f\in G$ that maps the $n$th term of the former of these two sequences to the $n$th term of the latter for each $n$.
A routine calculation now shows that  
$$\alpha h^{-2}h^{f}=\lambda, \quad \beta h^{-2}h^{f}=\delta,\quad \alpha h^{-3}h^{f}=\gamma, \;\;\;\hbox{and}\;\;\; \beta h^{-3}h^{f}=\gamma h^{-g}.$$

Let
$$ w_1:=[h^{-1},h^f],\quad \hbox{and} \quad
w_2:=[h^{-g},h^{gk}].$$ 
Further simple calculations show that $$\lambda w_1=\gamma\quad \hbox{and}\quad \lambda w_2=\delta.
$$
Moreover, $$\gamma w_2= \gamma \quad \hbox{and}\quad  \delta w_1=\beta h^{-3}h^f=
\gamma h^{-g} .$$  Hence
$$\lambda w_1w_2=\gamma\neq \gamma h^{-g}=\delta w_1 = \lambda w_2 w_1.$$ \end{proof}

The conclusion of Lemma \ref{primd} also holds for periodic o-primitive transitive $\ell$-permutation groups.

\begin{lemma} \label{primper}
Let $(G,\Omega)$ be periodic and o-primitive with period $t$.
Suppose that $g,h\in G$ with $\suppo(h)\cap \suppo(h^g)=\emptyset$.
Then there are $f,k\in G$ such that $$[h^{-1},h^f][h^{-g},h^{gk}]\neq [h^{-g},h^{gk}][h^{-1},h^f].$$   
\end{lemma}

\begin{proof} 
Since $\suppo(h)\cap \suppo(h^g)=\emptyset$, for any $\eta\in \Omega$ the set $\suppo(h) \cap (\alpha, \alpha t)$ is a proper subset of $(\eta,\eta t)$, since
otherwise there is $n\in \Z_+$ with $\eta h^n>\eta t$, whence $\suppo(h)=\suppo(h^n)=\Omega$. Similarly, $\suppo(h^g)\cap (\eta,\eta t)$ is a proper subset of $(\eta,\eta t)$. Fix $\eta\in \Omega$.
We may assume that there are supporting intervals $\Delta_1, \Delta_2:=\Delta_1g$ of $h$ and $h^g$, respectively, such that
 $\Delta_1\cup \Delta_2\subseteq (\eta,\eta t)$ and, as before, $\delta_1<\delta_2$ for all $\delta_i\in \Delta_i$ ($i=1,2$). Without loss of generality, $\delta_1 h>\delta_1$ for all $\delta_1\in \Delta_1$, and  $\delta_2 h^g>\delta_2$ for all $\delta_2\in \Delta_2$.
We now argue exactly as before to find elements $f,k$ of $G$ whose restrictions to $(\eta,\eta t)$ satisfy the required non-equality condition. 
\end{proof}

\section{Centralizers}
Throughout this section and the next we assume that $(G,\Omega)$ is an ample transitive $\ell$-permutation group, and we write $T$, $\fK$
for its root system and spine. 
We assume in addition in this section that $\fK$ has no minimal element.

For each $\Delta\in T$ and $h\in Q_\Delta$, let  $$X_{h}:=\{ [h^{-1},h^g] \mid g\in G\} \quad \hbox{and}\quad
W_{h}=\bigcup\{X_{h^g}\mid g\in G,\; [X_{h},X_{h^g}]\neq 1\}.$$
Since $(\rst(\Delta))^g=\rst(\Delta g)$ commutes with $\rst(\Delta)$ for $g\notin\St(\Delta)$ we have
$$X_h\subseteq \rst(\Delta)\quad\hbox{and} \quad W_h\subseteq \rst(\Delta).$$

In the proof of the main result of this section, Corollary \ref{2centd}, we will use the following observation.\medskip 

\begin{remark}\label{trivial}  {\rm Let $(\Lambda,\leq)$ be a totally ordered set and $\frak S$ be a {\em finite} set of pairwise disjoint convex subsets of $\Lambda$. If $f\in \Aut(\Lambda,\leq)$ and $\frak S f=\frak S$, then $Sf=S$ for all $S\in \frak S$.}  \end{remark}  

\begin{lemma} \label{40}
Let $\Delta\in T$ and $h\in Q_\Delta$.  
\begin{enumerate}
\item[\rm(a)]
Let  $\Delta'\in T$ with $\Delta'\subset \Delta$ and $\Delta' h\neq \Delta'$, and let $g\in \rst(\Delta')$ with $g\neq 1$. 
\begin{enumerate} 
\item[\rm (i)] 
Then $[h^{-1},h^g]\neq 1$. In particular, $X_h\neq 1$.
\item[\rm (ii)] 
If $f\in G$ and $[[h^{-1},h^g],f]=1$, then $\Delta' f = \Delta'$. In particular,
if $f\in\C_G(X_h)$ then $\Delta'f=\Delta'$.\end{enumerate}
\item[\rm (b)] If $\beta\in \suppo(h)$ and $f\in G$, then either $\beta f= \beta$ or 
$[[h^{-1},h^g],f]\neq1$ for some $g\in\rst(\Delta)$.
\end{enumerate}
\end{lemma}

\begin{proof}  (a) The elements $g^{h^{-1}},g, g^{h}$ have disjoint supports contained in $ \Delta'h^{-1},\Delta'$ and $\Delta'h$ respectively, and so the restrictions of $[h^{-1},h^g]=g^{-h^{-1}}gg^{-h}g$ to these three sets are
conjugates of $g^{-1}$ and $g^2$ and are non-trivial.  Assertion (i) follows. An arbitrary conjugate
$[h^{-1},h^g]^f$ has non-trivial restrictions to the distinct o-blocks $ \Delta'h^{-1}f,\Delta'f$ and $\Delta'hf$, and so if the hypothesis of (ii) holds then
$\{ \Delta'h^{-1},\Delta',\Delta'h\}f=\{\Delta'h^{-1},\Delta',\Delta'h\}$. Thus $f$ must map each of $\Delta'h^{-1},\Delta',\Delta'h$ to itself, by Remark \ref{trivial}.
 
(b) Suppose that $\beta f\neq \beta$. By Remark \ref{ccc}, one of the convex congruences $V(\beta, \beta h), V(\beta,\beta f)$ contains the other. Let $\Delta'\in T$ be a non-singleton o-block that is strictly contained in the o-block containing $\beta$ for each of these congruences.  Then $\Delta'\subset \Delta$ and $\Delta' f\neq \Delta'$.
Choose $g\in Q_{\Delta'}$; then $g\in\rst(\Delta')$ and from (a)(ii) we have $[[h^{-1},h^g],f]\neq1$.
\end{proof}  

\begin{lemma} \label{centd}
Let $\Delta\in T$ and $h\in Q_{\Delta}$.  
\begin{enumerate}
\item[\rm(a)]  $\C_G(X_h)$ contains the pointwise stabilizer of $\Delta$ and is contained in the pointwise stabilizer of $\suppo(h)$.
\item[\rm(b)] $$ W_{h} = \bigcup\{ X_{h^g}\mid g\in \St(\Delta)\}.$$ 
\item[\rm(c)]
$\C_G(W_h)$ is the pointwise stabilizer of $\Delta$.
\end{enumerate}\end{lemma} 

\begin{proof} (a) The first inequality holds since $X_{h}$ moves only points in
$\Delta$. 

Let $f\in \C_{G}(X_{h})$.  Since $X_{h} \subseteq \rst(\Delta)$ we have
$X_h=X_h^f\subseteq \rst(\Delta)\cap \rst(\Delta f)$, and since
$\{\Delta g\mid g\in G\}$ partitions $\Omega$ and $X_h\neq1$ we have 
$\Delta f=\Delta$.  Thus $f\in \St(\Delta)$.
Let $\beta\in \suppo(h)$.  
Then $\beta f=\beta$ by Lemma \ref{40}(b) since $f\in \C_{G}(X_{h})$.

(b) Let $g\in G$.  
If $\Delta g \neq \Delta$, then  $\rst(\Delta)\cap \rst(\Delta g)=1$ and the elements of $X_h$ and $X_{h^g}$ have disjoint support. Thus $[X_h,X_{h^g}]=1$. Hence
$$W_{h} = \bigcup\{ X_{h^g}\mid g\in \St(\Delta), [X_h,X_{h^g}]\neq 1\}.$$ 

Now let $g\in \St(\Delta)$ and $k=h^g$. So $k\in Q_{\Delta}$.  

First suppose that there is some
$\Gamma\in \pi(\Delta)$ with $\Gamma h\neq \Gamma$ and
$\Gamma k\neq \Gamma$. 
Choose $\beta\in \Gamma$ with $V(\beta, \beta h)=\kappa(\Delta)$. Let
$\Delta'\in T$ with $\beta\in \Delta'\subset \Gamma$.
There is an element $y\in Q_{\Delta'}$
with $\beta y\neq \beta$. 
Choose $\Delta''\in T$ with $\beta\in \Delta''\subset \Delta'$ and $\Delta''\neq \Delta''y$.  Choose  $x\in Q_{\Delta''}$ with $\beta\in \suppo(x)$.
Then $a:=[k^{-1},k^{x}]=x^{-k^{-1}}xx^{-k}x\neq 1$ since $\Delta''k$ is disjoint from $\Delta''$. 
Now $\suppo(a)\subseteq \Delta''k^{-1}\cup \Delta''\cup \Delta''k$ and $\Delta''y \cap \Delta''=\emptyset$.
Since $y^{h^{-1}}, y,y^h$ have disjoint supports it follows that $[h^{-1},h^y]$ cannot commute with $a$.
Thus $[a,[h^{-1},h^y]]\neq 1$. 
But $[h^{-1},h^y]\in X_h$ and $a\in X_{k}$.
Hence $[X_h,X_{k}]\neq 1$ and $X_{k}\subseteq W_h$.

Now suppose instead that $\Gamma k=\Gamma$ or $\Gamma h=\Gamma$ for all $\Gamma \in \pi(\Delta)$. Then $(G(\Delta),\Delta)$ must be of type (II) or (III), and  Lemma \ref{primd} or Lemma \ref{primper} gives elements of $X_h$ and $X_k$ whose images in the group $(G_\Delta,\pi(\Delta))$ fail to commute. Again we conclude that $X_k\subseteq W_h$.

(c)   The pointwise stabilizer of $\Delta$ lies in $\C_G(W_h)$ 
since $W_h\subseteq \rst(\Delta)$.

Let $\delta\in \Delta$ and $\alpha\in \suppo(h)$.
Choose $g\in \rst(\Delta)$ with $\alpha g=\delta$. So $\delta\in \suppo(h^g)$.
By (a), $\C_G(X_{h^g})$ stabilizes $\suppo(h^g)$ pointwise and so fixes $\delta$.  Since $\C_G(W_h)\subseteq \C_G(X_{h^g})$ by (b), we conclude that 
$\C_G(W_h)$ fixes $\delta$.  The assertion follows.  \end{proof}
 
Write $\C^2_G$ as shorthand for $\C_G\C_G$.
By Lemma \ref{centd}(c), we have\medskip

\begin{cor}\label{2centd}  Let $\Delta\in T$. Then $\C_{G}^2(W_h)=\rst(\Delta)$ for each $h\in Q_{\Delta}$. In particular, $\C^2_G(W_h)$ is independent of the choice of $h\in Q_\Delta$:
$$\C_G^2(W_h)=\C_G^2(W_{h'})\quad \hbox{for all}\;\;  h,h'\in Q_{\Delta}.$$
 
\end{cor}

\begin{proof}
By Lemma \ref{centd}(c) we have $\rst(\Delta)\subseteq \C^2_G(W_h)$.
Let $g\in \C^2_G(W_h)$ and $g_0:=\dep(g,\Delta)\in \rst(\Delta)\subseteq \C_G^2(W_h)$.
Thus $f:=gg_0^{-1}\in \C_G^2(W_h)$. 
If $f\neq 1$, let $\alpha\in \suppo(f)$ and $\Delta'\in T$ be an o-block of $V(\alpha,\alpha f)$ with $\alpha\in \Delta'$. Let $\Delta''\subset \Delta'$ with $\alpha\in \Delta''$ and $\Delta''f\neq \Delta''$. But any $y\in Q_{\Delta''}$ belongs to $\C_G(W_h)$ by Lemma \ref{centd}(c), yet $[y,f]\neq 1$, a contradiction. 
Hence $f=1$ and $g=g_0\in\rst(\Delta)$.
\end{proof}

\begin{remark} \label{rd}
 {\rm $Q_{\Delta}^g=Q_{\Delta g}$ and $\rst(\Delta)^g=\rst(\Delta g)$ for all $\Delta\in T$. Thus, to within conjugacy, $\C^2_G(W_h)$ depends only on $\kappa(\Delta)$ and not on the particular o-block of $\kappa(\Delta)$.} 
\end{remark}

In Section 7, we use Corollary \ref{2centd} and Remark \ref{rd} to derive information about the first-order theory of the o-primitive $\ell$-permutation groups $(G_{\kappa(\Delta)},\Om_{\kappa(\Delta)})$ from the first-order  theory of the $\ell$-group $G$ (under appropriate added hypotheses).

We can detect the hypothesis of this section that $T$ has no no minimal elements using the subgroups $\C^2_G(W_h)$:

\begin{cor} \label{Ed}
 For every $\Delta\in T$ and $h\in Q_\Delta$, there is  $h'\in \rst(\Delta)$ with 
 $1\neq \C_G^2(W_{h'}) \subset \C^2_G(W_h)$.   
\end{cor}

Now let $\Delta,\Delta'\in T$ with $\Delta'\supset \Delta$. Let $h_{\Delta}\in Q_{\Delta}$. There is  $h_{\Delta'}\in Q_{\Delta'}$ such that  $\Delta\subset \suppo(h_{\Delta'})$. So $h_{\Delta}\in \rst(\Delta)=
\C^2_{G}(W_{h_{\Delta}})$ and $h_{\Delta'}\not\in \rst(\Delta)=\C^2_{G}(W_{h_{\Delta}})$.
Conversely, if $\rst(\Delta)\subset \rst (\Delta')$ for some $\Delta'\in T$, then $\kappa(\Delta)<\kappa(\Delta')$. 
Thus we can detect the partial order on $T$ (and hence the total order on $\fK$).

\begin{cor} \label{leKd} Let $\Delta,\Delta'\in T$ and $h\in Q_{\Delta}, h'\in Q_{\Delta'}$. Then
$\kappa(\Delta)<\kappa(\Delta')$ if and only if 
$\C^2_{G}(W_{h^g}) \subset \C^2_{G}(W_{h'})$ for some $g\in G$. 
\end{cor}

\section{A minimal modification}

Results like those in the previous section also hold for certain ample $\ell$-permut\-ation groups $(G,\Om)$ for which $T$ has minimal elements.  We shall assume \medskip

{\noindent$(\dagger)$ \em if $T$ has a minimal element $\Delta$ then for each bounded interval $\Lambda\subseteq \Delta$ the group $G$ has a non-trivial element $g$ with $\suppo(g)\subseteq \Lambda$. } \medskip

This hypothesis will be assumed for the remainder of the article.  Hence if there is a minimal o-primitive component it is necessarily of type (II) and o-2 transitive.

\begin{remark} {\rm Let $g$ be as in Hypotheisis $(\dagger)$.  Then either $g\vee 1$ or $g^{-1}\vee1$ is a non-trivial  element $g'$ of  $G_+$ with  $\suppo(g')\subseteq \Lambda$. } \end{remark} 

The results in the previous section can all be recovered subject to Hypothesis $(\dagger)$. 
This follows from the following observation.

\begin{remark} \label{analogues}  
 {\rm Let $(H,\Lambda)$ be an o-$2$ transitive $\ell$-permutation group,
and  let $\sigma_1,\sigma_2\in \Lambda$  and  $h\in H_+$ with
 $\suppo(h)\subseteq (\sigma_1,\sigma_2)$.
Let $\lambda <\alpha_1<\beta<\gamma<\alpha_2<\lambda h<\beta h$. 
By o-$4$ transitivity, there is $g\in H$ with $\sigma_i g=\alpha_i$ for $i=1,2$, $\beta g=\beta$ and $\beta hg=\gamma$.
Let $f=h^g$. Then $f\in G_+$, $\suppo(f)\subseteq(\alpha_1,\alpha_2)$ and $\beta f=\gamma$. } \end{remark}

We sketch the easy extensions of the results of Section 4 for the sake of completeness.

\begin{lemma} \label{50}
Let $\Delta\in T$ and $1<h\in Q_\Delta$. 
\begin{enumerate}   
\item[\rm (a] Let $g\in \rst(\Delta)_+$ and suppose that $\suppo(g)\subseteq(\alpha,\alpha h)$ for some $\alpha\in \suppo(h)$. 
\begin{enumerate} \item[\rm(i)]  $[h^{-1},h^g] \neq 1$; in particular  $X_h\neq 1$.
\item[\rm (ii)] 
If $f\in G$ and $[[h^{-1},h^g],f]=1$, then $\suppo(g^{h^i})f = \suppo(g^{h^i})$ for $i=0,\pm 1$. \end{enumerate}
\item[\rm (b)] If $\beta\in \suppo(h)$ and $f\in G$, then either $\beta f= \beta$ or 
$[[h^{-1},h^g],f]\neq1$ for some $g\in\rst(\Delta)_+$.
\end{enumerate}
\end{lemma}

\begin{proof}  (a) (i)
 The restriction of $c:=[h^{-1},h^{g}]=g^{-h^{-1}}gg^{-h}g$ to $(\alpha,\alpha h)$ is $g^2> 1$.  

(ii) For $i=0,\pm 1$ we have $\suppo(g^{h^i})\subseteq (\alpha h^i, \alpha h^{i+1})$. These intervals are pairwise disjoint and $c$ is positive only on $\suppo(g)$ which lies strictly between $\suppo (g^{h^{-1}})$ and $\suppo (g^h)$.  By Remark \ref{trivial}, if 
$[c,f]=1$ then $\suppo(g^{h^i})f= \suppo(g^{h^i})$ for $i=0,\pm 1$.

(b) Suppose that $\beta f\neq \beta$ for some  $\beta\in \suppo(h)$.  There is $g\in \rst(\Delta)_+$ having support $\Lambda$ containing $\beta$ but disjoint from $\Lambda f$ and $\Lambda h$.
This contradicts  (a) (ii). 
\end{proof}

\begin{lemma} \label{5centd}
Let $\Delta\in T$ and $h\in Q_{\Delta}$. 
\begin{enumerate}
\item[\rm(a)] $\C_G(X_h)$ contains the pointwise stabilizer of $\Delta$ and is contained in the pointwise stabilizer of $\suppo(h)$.
\item[\rm(b)] $$ W_{h} = \bigcup\{ X_{h^g}\mid g\in \St(\Delta)\}.$$ 
\item[\rm(c)]
$\C_G(W_h)$ is the pointwise stabilizer of $\Delta$.
\end{enumerate}\end{lemma}

\begin{proof} The proofs of (a) and (c) are identical to those of Lemma \ref{centd}. The same is true for (b) if there are $\Delta'',\Delta'\in T$ with $\Delta''\subset \Delta'\subset \Delta$.
So assume that $\Delta$ is minimal in $T$ or covers a minimal element of $T$.
First assume that $\Delta$ is minimal in $T$.

 Let $g\in G$.  
If $\Delta g \neq \Delta$, then  $\rst(\Delta)\cap \rst(\Delta g)=1$ and the elements of $X_h$ and $X_{h^g}$ have disjoint support. Thus $[X_h,X_{h^g}]=1$. Hence
$$W_{h} = \bigcup\{ X_{h^g}\mid g\in \St(\Delta), [X_h,X_{h^g}]\neq 1\}.$$ 

Now let $g\in \St(\Delta)$ and $k=h^g$. So $k\in Q_{\Delta}$.  

First suppose that there is some
$\delta\in \Delta$ with $\delta \in \suppo(h) \cap \suppo(k)$.
Since $(G,\Om)$ satisfies $(\dagger)$, there is $y\in \rst(\Delta)_+$ with $\delta\in \suppo(y)\subseteq (\delta, \min\{\delta h,\delta k\})$.
We can find  $x\in G_+$  and $\beta\in \suppo(h)$ with $\beta\in \suppo(x)\subseteq (\delta,\delta y)$.
Then $a:=[k^{-1},k^{x}]=x^{-k^{-1}}xx^{-k}x\neq 1$ since $x^{k^{-1}}, x, x^k$ have disjoint supports.  Write $b:=[h^{-1},h^y]$. 
Since $y^{h^{-1}}, y,y^h$ have disjoint supports it follows that $[a,b]\neq1$ since $\beta a^b=\beta$ does not equal $\beta x^2=\beta a$.
But $a\in X_{k}$ and $b\in X_h$.
Hence $[X_h,X_{k}]\neq 1$ and $X_{k}\subseteq W_h$.

Now suppose that each element of $\Delta$ is fixed by $h$ or $k$. Since the minimal o-primitive component is of type (II),   Lemma \ref{primd} applies and provides elements of $X_h$ and $X_k$ whose images in the minimal o-primitive component fail to commute.
This completes the proof of (b) in the case when $\Delta$ is minimal in $T$.  An easy adaptation gives the proof in the case when $\Delta$ covers a minimal element of $T$. Hence the lemma is proved.
\end{proof}

The conclusions of Corollaries \ref{2centd} and \ref{leKd} and Remark \ref{rd} all follow easily using Lemma \ref{5centd} when $(G,\Om)$ satisfies $(\dagger)$. So does the conclusion of Corollary \ref{Ed} if $\Delta$ is not minimal in $T$.
We illustrate with one proof.

\begin{cor}\label{2centm}
For every $\Delta\in T$ and $h\in Q_{\Delta}$, 
$\C_{G}^2(W_h)=\rst(\Delta).$    Thus
if $\beta\in \Delta$ and $h'\in \rst(\Delta)$  with $V(\beta,\beta h')=\kappa(\Delta)$, then $$\C_G^2(W_h)=\C_G^2(W_{h'}).$$ In particular, 
$$\C_G^2(W_h)=\C_G^2(W_{h'})\quad \hbox{for all}\;\;  h,h'\in Q_{\Delta}.$$
\end{cor}

\begin{proof} 
By Lemma \ref{5centd}(c) we have $\rst(\Delta)\subseteq \C^2_G(W_h)$.
Let $g\in \C^2_G(W_h)$ and $g_0:=\dep(g,\Delta)\in \rst(\Delta)\subseteq C_G^2(W_h)$.
Thus $f:=gg_0^{-1}\in \C_G^2(W_h)$. 
If $f\neq 1$, let $\alpha\in \suppo(f)$ and $\Delta'\in T$ be a minimal o-block with $\alpha\in \Delta'$. Let $y\in \rst(\Delta')$ have support that is contained in $(\alpha,\alpha f)$. Then $y\in \C_G(W_h)$ by Lemma \ref{5centd}(c), yet $[y,f]\neq 1$, a contradiction. 
Hence $f=1$ and any element of $\C_{G}^2(W_h)$ must be contained in $\rst(\Delta)$.
\end{proof}

We can recognise whether an ample transitive $\ell$-permutation group $(G,\Omega)$ has a minimal o-primitive component of type (I). 
         
\begin{lemma} \label{abel}
Let $(G,\Omega)$ be an ample transitive $\ell$-permutation group. Then $(G,\Omega)$ has a minimal o-primitive component of type {\rm(I)} if and only if there is $h_1\in G_+$ such that $[a,b]=1$ for all $a,b\in G$ with $|a|\vee |b|\leq h_1$.
In this case, if $\Delta$ is a minimal o-block of $(G,\Omega)$ and $h_1\in Q_\Delta$, then $x\in \rst(\Delta)$ if and only if there is $h\in Q_\Delta$ with $h\geq h_1$ such that $(|x|\leq h)\;\&\;(\forall y)(|y|\leq h \longrightarrow [x,y]=1)$.
\end{lemma}
        
\begin{proof}
If $(G,\Omega)$ has a minimal o-primitive component of type (I), the condition clearly holds when $h_1\in \rst(\Delta)_+$ and $\Delta\in T$ is minimal.
Conversely, if such an  element $h_1\in G_+$ exists, then  $(G,\Omega)$ has no minimal o-primitive component that is either o-$2$ transitive or periodic but must have a minimal o-primitive component. The rest of the lemma is clear.
\end{proof}

\section{Interpretability of $\fK$ and $T$}
 
We shall now express our previous results in the first-order language of lattice-ordered groups.  The terms of the language are formed using the group and lattice operations $\cdot, ^{-1}, \vee, \wedge$.  
Formulae are obtained by taking conjunctions and disjunctions of equalities between terms and their negations, and quantifying over some subset of the variables appearing. To avoid confusion with the lattice operations, we write `$\&$' for conjunctions and `or' for disjunctions.

We need to be able to detect in this language the elements of the subgroups  $\rst(\Delta)$ and subsets $Q_\Delta$.  To handle the subsets $Q_\Delta$ we need to introduce another hypothesis.

We call $g\in G_+$ {\em irreducible}  if it satisfies the formula $\cB(g)$ given by
$$g>1\;\&\;(\forall g_1,g_2)((g_1\vee g_2=g\;\&\;g_1\wedge g_2=1)\longrightarrow (g_1=1\;\hbox{or}\; g_2=1)).$$
If $g_1\vee g_2=g$ and $g_1\wedge g_2=1$ then the open set $\suppo(g)$ is the union of the disjoint open sets $\suppo(g_1)$ and $\suppo(g_2)$.  Therefore $\cB(g)$ certainly holds if $g$ has a single interval of support, that is, if $g$ is a bump.

We note that if $G$ is abundant and $g$ is an irreducible element of
$\St(\Delta)_+$, then $g\in\rst(\Delta)$, because
$g_1\vee gg_1^{-1}=g$ and $g_1\wedge gg_1^{-1}=1$ where $g_1=\dep(g,\Delta)$.
Moreover, every irreducible element lies in $Q_\Delta$ for some $\Delta\in T$.

Let $(G,\Omega)$ be a transitive abundant $\ell$-permutation group that satisfies $(\dagger)$. Consider the following hypothesis:\medskip

\noindent $(\dagger\dagger)$ {\em each o-$2$ transitive component has irreducible elements.}\medskip

\medskip

For an example of an o-$2$ transitive $\ell$-permutation group where $(\dagger\dagger)$ fails, see \cite{GMacP}.

\begin{lem} \label{61} Let $(G,\Omega)$ be a transitive abundant $\ell$-permutation group that satisfies $(\dagger)$. Then $(G,\Om)$ satisfies $(\dagger\dagger)$ if and only if  for each $\Delta\in T$ there are irreducible elements in $Q_\Delta$.
\end{lem} 

\begin{proof}  If $\Delta\in T$ and $(G(\Delta),\pi(\Delta))$ is of 
type (I), then every element of $G(\Delta)_+$ has $\pi(\Delta)$ as a supporting interval and so is irreducible, whereas if  $(G(\Delta),\pi(\Delta))$ is of type (III) then, as noted in Section 2 there are elements $g\in G(\Delta)_+$ that have $\pi(\Delta)$ as a supporting interval and so are irreducible. 
Since $(G,\Om)$ satisfies Hypothesis $(\dagger\dagger)$, the group
$(G(\Delta),\pi(\Delta)$ has irreducible elements if it has type (II). It follows that for each $\Delta\in T$ there are irreducible elements $h\in Q_\Delta$.

Let $g$ be irreducible and $\delta\in\suppo(g)$. Let $\Delta$ be the o-block of 
$V(\delta,\delta g)$ containing $\delta$. Then $\Delta=\Delta g$ and so $\suppo(g)\subseteq \Delta$. By irreducibility we have $g\in Q_\Delta$.

The converse holds {\em a fortiori}.
\end{proof}

We will use the following formulae (cf.\ \cite[Section 4]{W}):
$$\begin{array}{rl}
\varphi(h,x)\colon &  \cB(h)\; \&\;(\exists y)(x=[h^{-1},h^{y}]);  \cr \cr
\psi(h,x)\colon &
(\exists t \exists y_1 \exists y_2)(\varphi(h,y_1)\;\&\;\varphi(h^t,y_2)\;\&\;\varphi(h^t,x)\;\&
\;\;  [y_1,y_2]\neq 1 );\cr\cr 

\gamma^1(h,x)\colon& (\forall y)(\psi(h,y) \rightarrow [x,y]=1); \cr\cr

\gamma(h,x)\colon&  (\forall y)(\gamma^1(h,y) \rightarrow [x,y]=1),\cr\cr

\eta(h_1,h_2)\colon&   \cB(h_1)\;\&\; \cB(h_2)\;\&\; (\C^2_G(W_{h_1})\subseteq \C_G^2(W_{h_2})).
\end{array}$$
Of course, the statement $\C^2_G(W_{h_1})\subseteq \C_G^2(W_{h_2})$ is shorthand for $(\forall x)(\gamma(h_1,x)\to \gamma(h_2,x))$. 
 By Corollary \ref{leKd}, if $h_i\in Q_{\Delta_i}$ ($i=1,2$) then $\eta(h_1,h_2)$ expresses in our language that $\Delta_1  \leq \Delta_2$ in the partially ordered set $T$.
 
\begin{remark} \label{irred1}
{\rm  Let $(G,\Om)$ be an abundant transitive $\ell$-permutation group satisfying $(\dagger)$. If $\mathcal B(h)$ holds, then $\C^2_G(W_h)=\{ x\mid \gamma(h,x)\}=\rst(\Delta)$. By Lemma \ref{61},  
$$(\dagger\dagger)\quad \hbox{is equivalent to}
 \quad G\models (\forall g)(\exists h)\gamma(h,g).$$} 
\end{remark} 

 For the rest of this section {\em $(G,\Omega)$ is a transitive abundant $\ell$-permutation group that satisfies $(\dagger)$ and
$(\dagger\dagger)$}. 

We can now give interpretations of $T$ and $\fK$ in $G$.

\begin{thm} \label{Bdef}  Write $B$ for the definable subset $\{h\mid \cB(h)\}$ of $G$ and $B_T$ the quotient of $B$ modulo the equivalence relation defined by the formula
$$\eta(h_1,h_2)\,\&\,\eta(h_2,h_1).$$  Then $B_T$ has the definable structure of  a partially ordered set and there is a definable $G$-equivariant bijection of partially ordered sets from $T$ to $B_T$.
\end{thm}

\begin{proof} 
If $h\in B$ holds, then $h\in Q_\Delta$ for a unique $\Delta\in T$ from Lemma \ref{61}.  By Corollaries \ref{2centd} and \ref{2centm} we have
$$X_h=\{ x\mid \varphi(h,x)\},\quad W_h=\{ x\mid \psi(h,x)\},\;\; \hbox{and}
\;\;  \C^2_G(W_h)=\{ x\mid \gamma(h,x)\}=\rst(\Delta).$$
Therefore the map $\{h\mid \cB(h)\}\to T$ determined by $h\mapsto \C_G^2(W_h)$
and the inverse of the bijection $\Delta\mapsto \rst(\Delta)$ yields a bijection 
$B\to T$.  Two irreducible elements $h_1,h_2$ of $B$ have the same image $\Delta$ if and only if 
both $h_1,h_2\in \rst(\Delta)$ and $\C_G^2(W_{h_1})=\C_G^2(W_{h_2})$; that is, if and only if the definable relation $\eta(h_1,h_2)\,\&\,\eta(h_2,h_1)$ holds. 
Therefore the quotient of $B$ modulo this relation provides an interpretation of the set $T$ in $G$.  This also gives an interpretation of the order on $T$, since for 
$\Delta_1,\Delta_2\in T$ with $h_i\in Q_{\Delta_i}$
we have $\Delta_1\subseteq \Delta_2$ if and only if $\C_G^2(h_1)\subseteq \C_G^2(h_2)$, or, in first-order words, if and only if $\eta(h_1,h_2)$ holds.  The correspondence is $G$-invariant since $\C_G^2(W_{h^g})=\C_G^2(W_h)^g$ for all $h\in B$, $g\in G$.

Now we interpret $\fK$.  This is in bijective order-preserving correspondence with $T$ modulo the relation $\sim$ that $\Delta_1\sim \Delta_2$ if and only if $\Delta_2=\Delta_1 g$ for some $g\in G$.  Thus $\fK$ is interpreted as $B$ modulo the definable equivalence relation 

$h_1\sim h_2:\quad (\exists g)\, (\eta(h_1^g,h_2)\,\&\,\eta(h_2,h_1^g))$. 
\end{proof}

We note some further consequences.  

\begin{cor} If the first-order theory of $G$ is decidable, then so is the first-order theory of $\fK$. \end{cor}

\begin{cor} 
Let $H$ be an arbitrary $\ell$-permutation group.  Then in all faithful representations of $H$ that are transitive, abundant, and satisfy $(\dagger)$ and $(\dagger\dagger)$, the spine and root system are uniquely determined. 
\end{cor}

\begin{cor}
\label{spined}  
Let $(H,\Lambda)$ be an abundant transitive $\ell$-permutation group that satisfies $(\dagger)$. If $G$ and $H$ are elementarily equivalent as lattice-ordered groups, then  $(H,\Lambda)$ satisfies $(\dagger\dagger)$ and $\fK_G$ and $\fK_H$ are elementarily equivalent as totally ordered sets.
\end{cor}

We define $\chi(h_1,h_2)$ to be the conjunction of the formulae
$$ h_1<h_2, \quad \eta(h_1,h_2),\quad
 \neg\eta(h_2,h_1),\quad (\exists x\neq 1)\gamma(h_1,x)$$ and
 $$(\forall h^*)([\eta(h_1,h^*)\;\&\;\neg\eta(h^*,h_1)]\longrightarrow \eta(h_2,h^*)).$$
  Clearly, if $\Delta_1,\Delta_2\in T$ with $h_i\in Q_{\Delta_i}$ ($i=1,2$), then 
 \begin{equation}\label{chi}
 \chi(h_1,h_2) \quad \hbox{if and only if}\quad \Delta_1\; \hbox{is covered by}\;\Delta_2.\quad  
 \end{equation}



 

Thus the first-order theory of $G$ allows us to determine if an element of $\fK$ has a predecessor or a successor, if there is a dense subset of the total order on $\fK$, or if $\mathcal U \in \fK$, the last since ${\mathcal U}\in \fK$ if and only if 
 $(\exists h_0)(\forall x)\gamma(h_0,x)$.  
 
\begin{cor} \label{lc2d}  \begin{enumerate}
\item[\rm(A)] $\Omega\in T$ if and only if there is $h_0\in G$ with $\cB(h_0)$ such that $\C_G^2(W_{h_0})=G$. 
\item[\rm(B)] Every element of $\fK$ has a sucessor if and only if for every irreducible $h\in G_+$ with  $\C_G^2(W_h)\neq G$, there is an irreducible $h'>h$  such that $\chi(h,h')$. 
\item[\rm(C)]  Every element of $\fK$ has a predecessor if and only if for every irreducible $h\in G_+$  there is an irreducible $h'\in G_+$ with $\chi(h',h)$.
\end{enumerate}\end{cor}

\begin{remark} \label{pred1}
{\rm If $K\in \fK$ has a predecessor in $\fK$, let $\Delta\in T$ be an o-block of $K$. Let $\Gamma\in \pi(\Delta)$; so $\Gamma=\C^2_G(W_h)$ for some $h\in G$ with $\mathcal B(h)$. Then $\pi(\Delta)=\{ \C^2_G(W_{h^g})\mid \chi(h,g)\}$. This set and its order are interpretable in our language since if $f,g\in G$ are such that $\chi(h,f)$ and $\chi(h,g)$, then $\C^2_G(W_{h^f})=\C^2_G(W_{h^g})$ if and only if $fg^{-1}\in \C^2_G(W_{h})$ and $\C^2_G(W_{h^f})<\C^2_G(W_{h^g})$ in $\pi(\Delta)$ if and only if $fg^{-1}\vee 1\in \C^2_G(W_{h})$ but $fg^{-1}\not\in \C^2_G(W_{h})$. Thus the first-order theory of the totally ordered set $\Om_K$ can be determined in our language for $\ell$-groups if $K\in \fK$ has a predecessor in $\fK$. }
\end{remark}

\section{Interpreting $o$-primitive sentences }

To motivate a definition, we prove a result in a special case.

\begin{lemma} \label{under} Let $(G,\Omega)$ be a transitive $\ell$-permutation group, 
 $\Delta \in T$ and $\Gamma \in \pi(\Delta)$.
If $g\in \rst(\Gamma)$ is irreducible, then there is $\Delta'\in T$ with $\Delta'\subseteq \Gamma$ such that $g\in \rst(\Delta')$.
\end{lemma}

\begin{proof}
Let $\mathcal C$ be the convex congruence having $\Gamma$ as an o-block. 
By Proposition \ref{cc},  $\mathcal C = \bigcup \{ K'\in \fK\mid K'<K\}$.
Let $g\in \rst(\Gamma)$ and $\gamma\in \suppo(g)$.
Then $K':=V(\gamma, \gamma g)\in \fK$ and $\Delta'$, the $K'$ o-block containing $\gamma$, is a subset of $\Gamma\subset\Delta$.
Hence $K'<K$.
But $g\in \rst(\Delta')$ since $g$ is irreducible.
The lemma follows.
\end{proof}

We would like to extend the lemma to all elements of $\rst(\Gamma)$. We say that a transitive $\ell$-permutation group $(G,\Omega)$ is {\em controlled}  if for every non-minimal $\Delta\in T$ and $g\in \rst(\Delta)$, either $V(\delta,\delta g)=\kappa(\Delta)$ for some $\delta\in \Delta\cap \suppo(g)$ or  there is $\Delta(g)\in T$ such that $\suppo(g)\subseteq \Delta(g)\subset \Delta$.  So $(G,\Omega)$ is controlled  if and only if $$\rst(\Gamma)=\bigcup \{\rst(\Delta')\mid \Delta'\in T,\;\Delta'\subseteq \Gamma\} $$
for all $\Delta\in T$ and $\Gamma\in \pi(\Delta)$ (cf. Lemma \ref{under}).

\begin{remark} \label{predcont}                    
{\rm  Let $(G,\Omega)$ be an abundant transitive  $\ell$-permutation group with every non-minimal element of $\fK$ having a predecessor. Then $(G,\Omega)$ is controlled since $\Gamma$ itself is a covering o-block
and Remark \ref{pred1} applies}.
 \end{remark}
 
 \begin{remark} \label{autcont}
 {\rm If  $(\Aut(\Omega,\leq),\leq)$ is transitive and controlled, then every non-minimal element of $\fK$ has a predecessor. For if $K\in \fK$ has no predecessor, let $\{K_\nu\mid \nu<\kappa\}$ be a strictly increasing sequence in $\fK$ indexed by the ordinals less than $\kappa$, say, with $\pi(K)=\bigcup_{\nu<\kappa} K_\nu$. If $\Delta$ is an o-block of $K$ and $\Gamma\in \pi(\Delta)$, let $\Delta_0\subset \Gamma$ be an o-block of $K_0$ and for $\nu>0$, let $\Delta_\nu\subset \Gamma$ be an o-block of $K_\nu$ with $\Delta_\nu$ disjoint from  $\bigcup_{\nu'<\nu} \Delta_{\nu'}$. Let $h_\nu\in Q_{\Delta_\nu}$ ($\nu <\kappa$) and  $h\in \Aut(\Omega,\leq)$ be defined by its restriction to $\Delta_\nu$ is $h_\nu$ and is the identity on $\Omega\setminus \bigcup_{\nu<\kappa} \Delta_\nu$. Then $h\in \rst(\Gamma)\setminus \bigcup \{\rst(\Delta')\mid \Delta'\in T,\;\Delta'\subseteq \Gamma\}$. }
 \end{remark}
  
Let $\mu(h,g)$ is the first-order formula for $\C^2_G(W_h)=\C^2_G(W_g)$ and $\vartheta(h',h)$ be the formula
 $\eta(h',h)\;\&\; \neg \eta(h,h')$. From the definition, we immediately obtain

\begin{cor}\label{theta}
\begin{enumerate}
\item[\rm (i)]  $(G,\Omega)$ is controlled if and only if $$G\models (\forall g)(\forall h)(\gamma(h,g)\longrightarrow (\mu(h,g)\;\hbox{or}\;(\exists h')(\vartheta(h',h)\;\;\&\;\;\gamma(h',g))).$$ 
\item[\rm (ii)] Let $(G,\Omega)$ be a controlled  $\ell$-permutation group, $\Delta\in T$ be non-minimal and $\Gamma\in \pi(\Delta)$. If $h\in Q_\Delta$, then $g\in \rst(\Gamma)$ if and only if
$$G\models (\exists h')(\vartheta(h,h')\;\;\&\;\;\gamma(h',g)).$$
\end{enumerate}
\end{cor}

Let $h\in Q_{\Delta}$ be irreducible and $\bar x=\{x_1,\dots,x_n\}$ be a finite set of variables. First replace $u(\bar x)=v(\bar x)$ by $u(\bar x)v(\bar x)^{-1}=1$ and $u(\bar x)\neq v(\bar x)$ by $u(\bar x)v(\bar x)^{-1}\neq 1$. Next replace
$$t(\bar x)=1\quad \hbox{by}\quad (\exists h')(\vartheta(h',h)\;\;\;\&\;\;\; \gamma(h',t(\bar x))),$$ and
$$t(\bar x)\neq 1\quad \hbox{by}\quad (\forall h')(\vartheta(h',h)\longrightarrow\neg\gamma(h',t(\bar x))).$$

For any formula $\rho(\bar x)$ free in $\bar x$, let $\rho_h^*(\bar x)$ be the result of replacing each basic subformula of $\rho$ as above. 
For a formula $$\sigma(x_1,\dots,x_{j-1},x_{j+1}, \dots, x_n):\equiv\exists x_j\rho(\bar x),$$ let
$$\sigma^*_{j,h}(x_1,\dots,x_{j-1},x_{j+1}, \dots, x_n) :\equiv \exists x_j (\gamma(h,x_j) \;\;\&\;\; \rho^*_h(\bar x)),$$
and for $\sigma(x_1,\dots,x_{j-1},x_{j+1}, \dots, x_n):\equiv \forall x_j \rho(\bar x)$, let
$$\sigma^*_{j,h}(x_1,\dots,x_{j-1},x_{j+1}, \dots, x_n) :\equiv \forall x_j (\gamma(h,x_j) \rightarrow \rho^*_h(\bar x)).$$
For example, if  $\alpha\in \Omega$ and $h\in Q_{\Delta}$ with $\kappa(\Delta)=V(\alpha,\alpha h)$, then we can express that $G(\Delta)$ is abelian by
modifying the sentence $\sigma:\equiv (\forall f,g)[f,g]=1.$ 
In this case, $\sigma^*_h$ is the sentence $$(\forall f,g)([\gamma(h,f)\;\&\;\gamma(h,g)]
\rightarrow (\exists h')(\vartheta(h',h)\;\&\; \gamma(h',[f,g])),$$ and $G(\Delta)\models \sigma$ if and only if $G\models \sigma^*_h$.

By the lemmata and corollaries from Sections 2, 4, 5 and the above, we obtain the following result.

\begin{prop}\label{primitd}
Let $(G,\Omega)$ be an abundant  transitive $\ell$-permutation group that satisfies $(\dagger)$ and $(\dagger\dagger)$ and is controlled. Then
for any $h\in Q_{\Delta}$ with $\cB(h)$ and sentence $\sigma$ of the language,
$$G(\Delta)\models \sigma\quad \hbox{if and only if}\quad G\models \sigma^*_h.$$
\end{prop}

By Corollary \ref{theta}(i) and Remark \ref{irred1}, if $(G,\Omega)$ is controlled then so is any abundant transitive   $\ell$-permutation group $(H,\Lambda)$ satisfying $(\dagger)$ and $(\dagger\dagger)$ with $H\equiv G$. Hence

\begin{cor} \label{oprimsent}
Let $(G,\Omega)$ and $(H,\Lambda)$ be abundant  transitive $\ell$-permutation groups that satisfy $(\dagger)$ with $(G,\Om)$ satisfying $(\dagger\dagger)$ and $H\equiv G$. If $(G,\Om)$ is controlled, then so is $(H,\Lambda)$ and every sentence that holds in some o-primitive component of $(G,\Om)$ also holds in some o-primitive component of $(H,\Lambda)$. Indeed, if $(G(\Delta_i),\pi(\Delta_i))$ are  o-primitive components of $(G,\Om)$ with $\Delta_1\leq \cdots \leq \Delta_r$ and $G(\Delta_i)\models \sigma_i$ for $i=1,\dots,r$, then there are o-primitive components $(H(\Delta'_i),\pi(\Delta_i'))$  of $(H,\Lambda)$ with $\Delta'_1\leq \cdots \leq \Delta'_r$ and $H(\Delta'_i)\models \sigma_i$ for $i=1,\dots,r$.  
\end{cor}

We write $\Sigma$ for the (countable) set of sentences of the first-order language of $\ell$-groups.

\begin{defn}{\em
A {\em coloured chain} is a totally ordered set $(C,\leq)$ together with a countable set $\{P_i\mid i\in I\}$ of unary predicates called {\em colours} , such that for all $c\in C$ the set $\{i\mid P_i(c)\}$ is non-empty.
One says that $c$ has colour $i$ if $P_i(c)$ holds.  (Elements of $C$ are allowed to have more than one colour.)

The coloured chain $\frak C_G$ associated with an $\ell$-permutation group $(G,\Om)$ is the totally ordered set $\fK_G$ together with the set $\{P_\sigma\mid\sigma\in \Sigma\}$, such that $P_\sigma(K)$ if and only if $G_K\models \sigma$ (for $K\in\fK$, $\sigma\in\Sigma$)}.  
\end{defn}

We note that if $\rho,\tau\in \Sigma$ and $F\models \rho \longrightarrow \tau$ for every $\ell$-group $F$, then $P_\rho(c)\longrightarrow P_\tau(c)$ is an axiom for the theory of coloured chains of $ell$-groups.
 
\begin{proof}[Proof of Theorem \ref{MTd}]  By Remark \ref{irred1}, to complete the proof of Theorem \ref{MTd}, we only need to prove the assertion about $\frak C_G \equiv \frak C_H$.
This follows immediately because $(H,\Lambda)$ is controlled by Corollary \ref{theta}(i) and if $h\in Q_\Delta$ is irreducible where $\Delta\in T$ is an o-block of $K$, then $\frak C_G\models P_\sigma(K)$ if and only if
$G\models \sigma^*_h$ by Proposition \ref{primitd}.
 \end{proof}

\begin{cor} \label{oprimsame}
Let $(G,\Omega)$ and $(H,\Lambda)$ be abundant  transitive $\ell$-permutation groups that satisfy $(\dagger)$ with $(G,\Om)$ satisfying $(\dagger\dagger)$ and $H\equiv G$. If $(G,\Om)$ is controlled and all o-primitive components $G(\Delta)$ are elementarily equivalent, then the same is true of all o-primitive components of $H$; all are elementarily equivalent to $G(\Delta)$. 
\end{cor}

\begin{proof} This follows at once since $$\frak C_G\models (\forall K_1,K_2) (P_\sigma(K_1)\longleftrightarrow P_\sigma(K_2))$$ for all $\sigma\in \Sigma$,
and all $K_1,K_2$ can be obtained through all irreducible $h_1,h_2$.
\end{proof}

Since an o-primitive component is of type (I) if and only if it is abelian,
and since non-dense subgroups of $\RR$ are cyclic,  we deduce

\begin{cor} \label{Z}
 Let $(G,\Omega)$ and $(H,\Lambda)$ be abundant  transitive $\ell$-permutation groups that satisfy $(\dagger)$ and $(\dagger\dagger)$ with $H\equiv G$.  Suppose that $(G,\Om)$ is controlled. If all o-primitive components of $(G,\Om)$ are abelian  then so are those of $(H,\Lambda)$; if all o-primitive components of $(G,\Om)$ are cyclic then so are those of  $(H,\Lambda)$.
\end{cor}

\begin{cor} \label{fin}
Let $(G,\Omega)$ and $(H,\Lambda)$ be abundant  transitive $\ell$-permutation groups that satisfy $(\dagger)$ and $(\dagger\dagger)$ and $H\equiv G$. If $\fK_G$ is finite with $r$ elements $K_1<\dots<K_r$, then $\fK_H$ is finite with $r$ elements $K'_1<\dots<K'_r$, and $G_{K_i}\equiv H_{K'_i}$ for each $i$. 
\end{cor}

By an easy extension of Lemma \ref{abel} to soluble groups of fixed derived length, we can determine whether $\frak K$ is finite and each o-primitive component is abelian whether or not $(\dagger)$ holds.

\begin{cor} \label{finab}
Let $(G,\Omega)$ and $(H,\Lambda)$ be abundant  transitive $\ell$-permutation groups with every o-primitive component of $(G,\Om)$ abelian. If $\fK_G$ is finite with $r$ elements $K_1<\dots<K_r$ and $H\equiv G$, then $\fK_H$ is finite with $r$ elements $K'_1<\dots<K'_r$, and $G_{K_i}\equiv H_{K'_i}$ for each $i$. 
\end{cor}

Suppose that $(G,\Omega)$ is a transitive abundant $\ell$-permutation group that is controlled and does not have a minimal abelian o-primitive component. If $\fK_G=\{K_n\mid n\in \Z_-\}$ where $K_n>K_{n-1}$ for all $n\in \Z_-$, we can determine whether $\mathcal U\in \fK_G$ and determine the `level' of an o-primitive component of $(G,\Om)$:

\begin{cor} \label{Z-}
Let $(G,\Omega)$ and $(H,\Lambda)$ be transitive abundant $\ell$-permutation groups that satisfy $(\dagger)$ and $(\dagger\dagger)$. If $\fK_G=\{K_n\mid n\in \Z_-\}$ where $K_n>K_{n-1}$ for all $n\in \Z_-$ and $G\equiv H$, then $\fK_H=\{K'_n\mid n\in \Z_-\}\cup \fK'$ for some totally ordered set $\fK'$ with $K'_n>K'_{n-1}>K'$ for all $n\in \Z_-$ and $K'\in \fK'$ and $G_{K_n}\equiv H_{K'_n}$ for all $n\in \Z_-$. In particular, if $\fK'=\emptyset$, the first-order theory of the o-primitive component of $(H,\Lambda)$  at level $n$ is equal to that of $(G,\Om)$ at level $n$ for each $n\in \Z$. 
\end{cor}

Finally, we show that the converse of Theorem \ref{MTd} is false. 

\begin{example}
\label{nc}
{\rm Consider the $\ell$-permutation groups $G:=\Wreath \{(\Z,\Z)\mid \Z\}$ and $H:=\wreath \{(\Z,\Z)\mid \Z\}$.  They are both
abundant and transitive. They have  the same spines $\fK=\Z$ of convex congruences every element of which has a predecessor. Hence $G$ and $H$ are controlled. They both satisfy $(\dagger)$ and $(\dagger\dagger)$, and have the same o-primitive components (and so the same associated coloured chain).  Let $K_0\in \fK$.
By taking an element $g\in G$ that is translation by $1$ on every $K_0$-class and is $0$ on every non-$K_0$ o-block,  we see that $G$ satisfies the sentence $(\exists g>1)(\forall h>1)(g\wedge h>1)$.  However $H$ does not satisfy this sentence and so $G\not\equiv H$. }
\end{example}

\section{Applications}

First we illustrate the methods developed above with some applications to $\ell$-groups all of whose o-primitive components are abelian.  

The structure of transitive abelian o-primitive $\ell$-groups is  fully understood. Each such group $A$ is an o-group and either isomorphic to $(\Z,\Z)$ or dense in $\RR$ and determined by its set of Szmielew invariants $\{ \beta(p,A)\mid p\in P\}$, where $P$ is the set of primes and $\beta(p,A(K)):= \min\{ \dim(A/pA),\aleph_0\}$
 (see \cite{RZ} or \cite{Wei}).  Using this fact, we can list some consequences of our results for groups all of whose o-primitive components of $(G,\Omega)$ are abelian.

\begin{cor} \label{opabd} 
Let $(G,\Omega)$ be an abundant transitive  $\ell$-permutation group. If $(G,\Omega)$ is controlled with each o-primitive component abelian and the spine has no minimal element, then the same holds for any abundant transitive $\ell$-permutation group $(H,\Lambda)$ with  $H\equiv G$. Moreover, the ordered sets $\fK_G$ and $\fK_H$ satisfy the same first-order sentences and the o-primitive actions are given by right regular representations of subgroups of $\RR$.  If in addition $G$ has the property that each o-primitive component is isomorphic to $\Z$, then so does $H$.
\end{cor} 

The final assertion above holds by Corollary \ref{Z}.

By Corollary \ref{oprimsame} we have

\begin{cor}
\label{sumaut}
Let $\fK$ be an arbitrary totally ordered set with no minimal element.
Suppose that $A$ is a subgroup of $\RR$ and $A_{K}:=A$ for each $K\in \fK$.
Let $(W,\Lambda)$ be the wreath product 
$\wreath\{(A_K,A_K)\mid K\in \fK\}$.
Assume that $\fK'$ is a totally ordered set and that $\{(G_{K'},\Omega_{K'})\mid K'\in \fK'\}$ is a family of  o-primitive $\ell$-permutation groups. Define $(G,\Omega):=\wreath\{(G_{K'},\Omega_{K'})\mid K'\in \fK'\}$.
If $G\equiv W$, then $G_{K'}\equiv A$ for each $K'$ and $\fK',\fK$ are elementarily equivalent as totally ordered sets.
\end{cor}

When the hypotheses of the above corollary hold and additionally every element of $\fK$ has a predecessor, then the corresponding result holds for Wreath products by the same proof using Remark \ref{predcont}. This immediately implies Corollary \ref{sumauto}.

\begin{cor} \label{ARd}
Let $\fK$ and $\fK'$ be totally ordered index sets and $\{(\Omega_K,\leq_K)\mid K\in \fK\}$ be a family of $2$-homogeneous totally ordered sets. If $\wr\{(\Aut(\Omega_K,\leq_K),\Omega_K)\mid K\in \fK\}\equiv \wr\{(\Aut(\RR,\leq),\RR)\mid K'\in \fK'\}$, then each $(\Omega_K,\leq_K)$ is order-isomorphic to $(\RR,\leq)$ and  $\fK$ and $\fK'$ satisfy the same first-order sentences in the language of ordered sets. The same conclusion holds with $\wr$ replaced by $\Wr$ if every element of $\fK'$ has a predecessor.
\end{cor}

 All proper convex congruences belong to $\fK$ if the root system $T$ is a (path-connected) tree in the traditional sense; that is, if between any two distinct $\Delta_1,\Delta_2\in T$, there is a unique {\em finite} path between $\Delta_1$ and $\Delta_2$ without loops. In this case, $\fK$ is order-isomorphic to $\Z_+$, $\Z_-$, $\Z$ or a finite totally ordered set. By Proposition \ref{spined}, each of these possibilities can be distinguished by sets of sentences in the first-order language of $G$ if $(G,\Omega)$ satisfies $(\dagger)$ and $(\dagger\dagger)$. 

\begin{cor} Let $(G,\Omega)$ and $(H,\Lambda)$ be abundant transitive $\ell$-permutation groups that satisfy $(\dagger)$ and $(\dagger\dagger)$ with  $(G,\Omega)$  controlled. Assume further that $T_G$ and $T_H$ are path-connected trees. If $G\equiv H$, then the totally ordered sets $\fK_G$ and $\fK_H$ are order-isomorphic and the corresponding o-primitive components of $G$ and $H$ are elementarily equivalent.
\end{cor}

Hence we obtain

\begin{cor} Let $(\Omega,\leq)$ be the totally ordered set $\prod_{n=-1}^{-\infty} \RR$ ordered lexicographically. Let $(\Lambda,\leq)$ be a totally ordered set with $\Aut(\Lambda,\leq)$ transitive.  If $\Aut(\Lambda,\leq)\equiv \Aut(\Omega,\leq)$, then $\fK_{\Aut(\Lambda,\leq)}\equiv \Z_-$ and every  o-primitive component $\Aut(\Lambda,\leq)$ is $(\Aut(\RR,\leq),\RR)$. If, additionally, $T_{\Aut(\Lambda,\leq)}$ is path-connected, then $\fK_{\Aut(\Lambda,\leq)}$ is order-isomorphic to $\Z_-$.
\end{cor}

\noindent {\bf Acknowledgment.} {\rm This research was begun when the second author was the Leibniz Professor at the University of Leipzig. Both authors are most grateful to the Research Academy, Leipzig and the Leibniz Program of the University of Leipzig for funding a visit by the first author that made this research possible.}

\bigskip

\noindent A. M. W. Glass,

\noindent QUEENS' COLLEGE, CAMBRIDGE CB3 9ET,
U.K.\medskip

E-mail: amwg@dpmms.cam.ac.uk \bigskip

\noindent John S. Wilson,

\noindent MATHEMATICAL INSTITUTE, ANDREW WILES BUILDING, WOODSTOCK ROAD, OXFORD OX2 6GG, ENGLAND

E-mail: John.Wilson@maths.ox.ac.uk


\bigskip



\begin{thebibliography}{99}

\bibitem{G81} A. M. W. Glass, {\em Ordered Permutation Groups}, London Math. Soc. Lecture Notes Series {\bf 55}, Cambridge University Press, Cambridge, 1981.

\bibitem{G99} A. M. W. Glass, {\em Partially Ordered Groups}, Series in Algebra {\bf 7}, World Scientific Pub. Co., Singapore, 1999.

\bibitem{GGHJ}
A. M. W. Glass, Y. Gurevich, W. C. Holland, M. Jambu-Giraudet,
Elementary theory of automorphisms of doubly homogeneous chains, in
{\em Logic Year 1979-80, University of Connecticut}, (ed. M. Lerman, J. H. Schmerl, R. I. Soare), Springer Lecture Notes 859, Heidelberg, 1981, pp. 67--82.

\bibitem{GMacP}
A. M. W. Glass, H. Dugald Macpherson, Permutation groups without irreducible elements,  (submitted).

\bibitem{GU}
Y. Gurevich, Expanded theory of ordered abelian groups, Annals Math. Logic {\bf 12} (1977), 193--228.

\bibitem{GH}
Y. Gurevich, W. C. Holland, Recognizing the real line, Trans.\ American\ Math.\ Soc. {\bf 265} (1981) 527--534.


\bibitem{HM}  
W. C. Holland, S. H. McCleary, Wreath products of ordered permutation groups, Pacific J. Math. {\bf 31} (1969), 703--716.

\bibitem{Mc}
S. H. McCleary, o-primitive ordered permutation groups I, Pacific J. Math. {\bf 40} (1972), 349--372;  o-primitive ordered permutation groups II, {\em ibid} {\bf 49} (1973), 431--443.


\bibitem{RZ}
A. Robinson, E. Zakon, Elementary properties of ordered abelian groups, Trans.\ American\ Math.\ Soc. {\bf 96} (1960), 222--236.

\bibitem{Wei}
V. Weispfenning, Model theory of abelian groups, in {\em Lattice-ordered Groups: Advances and Techniques}, (ed. A. M. W. Glass, W. C. Holland), Kluwer Acad. Pub., Dordrecht, 1989,  pp. 41--79.

\bibitem{W}
J. S. Wilson, The first-order theory of branch groups, J.\ Austral.\ Math.\ Soc., to appear.
\end{thebibliography}
\end{document}